\documentclass[11pt]{article}

\usepackage{amsmath}
\usepackage{amssymb}
\usepackage{amsthm}
\usepackage{xfrac}
\usepackage{amsfonts}
\usepackage{dsfont}
\usepackage[english]{babel}
\usepackage{bbm}
\usepackage{bm}
\usepackage{booktabs} 
\usepackage{caption}
\usepackage{cite}
\usepackage{enumerate}
\usepackage{epigraph}
\usepackage{etoolbox} 
\usepackage{float}
\usepackage{graphicx}
\usepackage{mathtools}
\usepackage{pgfplots}
\usepackage{float}
\usepackage{mathtools}
\usepackage{microtype} 
\usepackage{multicol}
\usepackage{gauss} 
\usepackage{geometry}
\usepackage{todonotes}
\usepackage[perpage,para,symbol*]{footmisc}
\usepackage{subfigure}
\usepackage{tabularx}
\usepackage{tikz}
\usepackage{url}
\usepackage{verbatim}
\usepackage{xparse}
\allowdisplaybreaks

\newcolumntype{x}[1]{!{\centering\arraybackslash\vrule width #1}}

\newtheorem{theorem}{Theorem}

\newtheorem{definition}{Definition}
\newtheorem{conjecture}{Conjecture}

\newtheorem{construction}{Construction}
\newtheorem{openproblem}{Open Problem}

\newcommand{\N}{{\mathbb N}}
\newcommand{\R}{{\mathbb R}}

\newcommand{\tr}{\text{tr}}

\NewDocumentCommand{\ceil}{s O{} m}{%
  \IfBooleanTF{#1} 
    {\lceil#3\rceil} 
    {#2\lceil#3#2\rceil} 
}

\makeatletter
\let\@fnsymbol\@arabic
\newcommand{\specificthanks}[1]{\@fnsymbol{#1}}
\makeatother

\begin{document}

\title{\textbf{Some open mathematical problems on fullerenes}}
\insert\footins{\footnotesize \hspace{-0.5cm}
\begin{tabular}{l @{\hspace{1cm}} l @{\hspace{1cm}} l}
    Artur Bille   & Victor Buchstaber & Evgeny Spodarev  \\
    \texttt{artur.bille@uni-ulm.de} &
    \texttt{buchstab@mi-ras.ru} & 
    \texttt{evgeny.spodarev@uni-ulm.de} 
\end{tabular}\\
}

\author{
Artur Bille\thanks{Ulm University, Germany \quad\quad *Corresponding author}\textsuperscript{\hspace{0.15cm}*}\and 
Victor Buchstaber\thanks{Steklov Mathematical Institute RAN, Russia} \and 
Evgeny Spodarev\textsuperscript{\specificthanks{1}}
}
\date{}
\maketitle

\begin{abstract}
\label{abstract}
\noindent Fullerenes are hollow carbon molecules where each atom is connected to exactly three other atoms, arranged in pentagonal and hexagonal rings. 
Mathematically, they can be combinatorially modeled as planar, 3-regular graphs with facets composed only of pentagons and hexagons.
In this work, we outline a few of the many open questions about fullerenes, beginning with  the problem of generating fullerenes randomly. 
We then introduce an infinite family of fullerenes on which the generalized Stone-Wales operation is inapplicable. 
Furthermore, we present numerical insights on a graph invariant, called \textit{character} of a fullerene, derived from its adjacency and degree matrices.
This descriptor may lead to a new method for linear enumeration of all fullerenes. 
\end{abstract}

\textbf{Secondary:} 05C10, 00A27, 05C30, 05C85.



\section{Introduction}\label{sec:introduction}
Fullerenes, spherical molecular structures composed of pentagonal and hexagonal rings of carbon atoms, have long captivated researchers across various disciplines.
Shortly after their discovery, the enumeration and counting of fullerenes became a subject of high interest among chemists, physicists, computer scientists and mathematicians which has led to a substantial body of research.
Despite decades of study, several fundamental questions regarding fullerenes, particularly from a mathematical perspective, remain unresolved.  
This paper aims to address some of these open problems by presenting hypotheses based on numerical evidence, which can be replicated using the algorithms provided in \cite{bille_github_phd}.
The insights and algorithms discussed here may serve as the groundwork for further research in the mathematical analysis of fullerenes. 
A notable example of a proof assisted by computer computations of a conjecture, long unresolved until recent computational advances, is the \textit{Barnette-Goodey conjecture}.
Originally stated in the 1960s, this conjecture asserts that all fullerene graphs are Hamiltonian.
In other words, every fullerene graph contains a closed path that passes through every vertex of the graph once without repeating any vertices, except for the starting and ending vertex.
In 2017, Kardo\v{s}, who had previously proven the sharpest lower bound for the number of vertices in the longest path on a fullerene graph \cite{erman09}, provided a valid computer-assisted proof of the Barnette-Goodey conjecture in \cite{computer_proof}.

A year before the Nobel Prize in Chemistry was awarded to Curl, Kroto and Smalley for their pioneering work on fullerenes, Fowler and Manolopoulos  \cite{FowlerManop} offered a comprehensive overview of the theory and mathematics behind fullerenes. 
This influential work outlined open fundamental mathematical questions related to fullerenes as well as their physical and chemical properties. 
Additionally, they introduced the important \textit{face spiral conjecture}, which offered the first general framework for generating and enumerating (mathematical) fullerenes. 

Graph theory has proven to be an indispensable tool for a detailed analysis of fullerene structures. 
In 2016, Andova et al. \cite{AndKarSkr16} provided an overview of graph invariants and their potential correlation with fullerene molecule stability. 
A subsequent study by Sure et al. \cite{Grimme17}, published just one year later, expanded this analysis from a chemical and physical standpoint, employing high-accuracy quantum chemistry methods to compute the relative energies of all $C_{60}$ isomers.
The study drew connections between these results and various topological indices and geometrical measures, ultimately proposing a list of \textit{good stability criteria}.
In \cite{Bille20}, the \textit{Newton polynomials} of (subgraphs of) dual fullerene graphs were introduced as another good stability criterion, based on the spectra of the adjacency matrix $A$ of these graph representations. 
In the present paper, we extend these results to incorporate the degree matrix $D$ and their linear combinations $\alpha A+\beta D$, and introduce a novel graph invariant, termed the \textit{character} of fullerenes, which has the potential to address some of the open problems outlined here and in \cite{FowlerManop}.

This paper is structured as follows.
First, we revisit and extend the concept of the face spiral in the context of methods for constructing fullerenes at random, drawing on the approaches discussed in \cite{Plestenjak96}, and review additional algorithms for constructing fullerene isomers, including the \textit{Stone--Wales operation} and two of its generalizations, one of which will be examined in more detail. 

Subsequently, we present novel findings on the eigenvalues of linear combinations of adjacency and degree matrices representing (sub)graphs of fullerenes.
Using these eigenvalues, we introduce a graph invariant, which allows for a linear ordering of fullerenes within $C_n$ and across all feasible values of $n$. 
This ordering provides an alternative approach enumerating fullerenes, compared to the commonly accepted method based on the face-spiral hypothesis.

\section{Preliminaries}\label{section:preliminaries}
We begin by recalling important definitions and notation used throughout this article.

A \textit{fullerene} is a convex, 3-regular polytope consisting only of pentagonal and hexagonal facets.
Let $n$ denote the number of vertices in a fullerene.
By Euler's polyhedron formula and Eberhard's theorem \cite{Grunbaum}, every fullerene contains exactly 12 pentagonal facets, while the number of hexagons $\sfrac{n}{2}-10$ depends solely on $n$.
Consequently, $n$ must be an even integer for at least one fullerene with $n$ vertices to exist. 
Notably, no fullerene with $n=22$ vertices exists, as shown in \cite{Grünbaum_Motzkin_1963}.
Thus, we define a positive even integer $n\geq 20$, $n\not=22$, as \textit{feasible}. 

Let $C_n$ denote the set of all fullerenes with exactly $n$ vertices. 
This set can be partitioned into equivalence classes under graph isomorphisms, with each class referred to as a \textit{$C_n$-isomer}.
Additionally, let $\overline{C_{n}}$ represent the union of all $C_{\Tilde{n}}$ with $\Tilde{n}\leq n$.

By Steinitz's theorem \cite{Grunbaum}, fullerenes can be associated with planar graphs that preserve their combinatorial structure and properties. 
A graph $G$ is defined as a tuple $G=(V(G),E(G))$, where $V(G)$ is the set of vertices, and $E(G)$ the set of edges.
Each edge is an unordered pair of vertices $(v,w)$ with $v,w\in V(G)$.
A \textit{facet} of the graph $G$ is a connected region of the plane enclosed by a \textit{closed path} of vertices $(v_1,\ldots,v_l)$, $l\geq 2$, where $\left(v_1,v_l\right)$, $\left(v_j,v_{j+1}\right)\in E(G)$ for $j=1,\ldots,l-1$.
We define one facet as being \textit{larger} than another if it is enclosed by more vertices.
Additionally, a vertex $v\in V(G)$ is called a \textit{boundary vertex} of a facet $f$ if $v$ is part of the closed path enclosing $f$.

A common graph representation of polytopes, like fullerenes, is given by the \textit{Schlegel diagram}, being a projection of a 3-dimensional polyhedron onto the plane.
Although this projection is not bijective (meaning that the resulting embedding of the graph is not unique, where uniqueness or equality ($\simeq$) of graphs is meant up to an isomorphism), the combinatorial properties of the graph are invariant.
Henceforth, we identify a \textit{fullerene} with its corresponding planar graph representation.

Define $\text{iso}(n):=\left| C_n\right|$ and $\overline{\text{iso}(n)}:= \left| \overline{C_n} \right|$ to be the number of fullerenes in $C_n$ and $\overline{C_n}$, respectively.
The exact values of $\text{iso}(n)$ are known for $n\leq 400$ by generating all possible isomers, cf. \cite{HouseofGraphs}, though an explicit formula for $\text{iso}(n)$ remains unknown.
However, an asymptotic behaviour was derived by Rukhovich \cite{Rukh18}, building on general insights from Thurston \cite{Thurs98} and Engel et al. \cite{engel2018}:

    \begin{align*}
        \liminf_{n\xrightarrow{}\infty} \frac{\text{iso}(n)}{n^9} = \frac{809}{2^{15}\cdot 3^{13}\cdot 5^2},\qquad
        \limsup_{n\xrightarrow{}\infty} \frac{\text{iso}(n)}{n^9} = \frac{809}{2^{15}\cdot 3^{13}\cdot 5^2} \zeta(9),
    \end{align*}
where $\zeta$ is the \textit{Riemann zeta function}, with $\zeta(9)\approx 1.00200839$.

Notably, Engel et al. \cite{engel2024} derived an explicit formula for counting the so-called \textit{oriented fullerenes} with $n$ vertices, which are typically far more numerous than their unoriented counterparts.
For example, they report 3532 oriented fullerenes with $60$ vertices, while it is commonly known that $\text{iso}(60)=1812$. 
Ordinary fullerenes differ from oriented fullerenes in that the set of edges of the latter is defined in terms of ordered pairs of vertices.

Since fullerenes are 3-connected, Whitney's theorem cf.\cite{Grunbaum} ensures that their dual graph is unique. 
A \textit{dual fullerene}, denoted by $T_n$, can also be viewed as an element of the set of all convex triangulations of the sphere, as described in \cite{engel2024}.
In this article, we restrict our focus to dual fullerenes only, noting that all results can easily be translated in terms of the original fullerene structures.
In presented figures, vertices are color-coded: green for vertices of degree five, red for vertices of degree six, and white for vertices of arbitrary degree.

The set of vertices $V(T_n)$ can be partitioned into two sets: one containing vertices of degree five and the other containing vertices of  degree six. 
The subgraph of $T_n$ induced by the vertices of degree five captures the connectivity between the pentagons, while the subgraph induced by vertices of degree six represents the connectivity between hexagons. 
We denote these graphs by $T_n^5$ and $T_n^6$, respectively.

Fullerenes with no edge connecting any pair of vertices in $T_n^5$ are said to satisfy the \textit{Isolated Pentagon Rule} (IPR).
These \textit{IPR-isomers} are believed to be chemically more stable than non IPR-isomers, making them typically more deserving the detailed analysis.

The introduced graphs can be uniquely represented by their adjacency matrix $A$. 
Additionaly, the degree matrix $D$ (a diagonal matrix with degrees of all vertices on its diagonal), is often considered in graph theory, along with their linear combinations $\alpha A + \beta D$, $\alpha,\beta \in \R$.
For example, choosing $\alpha = -1$ and $\beta = 1$ results in the well-known \textit{Laplacian matrix} of a graph.


\section{Random fullerenes}\label{section:random_fullerenes}
The need to randomly select a fullerene uniformly distributed on $C_n$ arises in many contexts,  particularly in mathematical research.
For example, a typical objective is to study the limiting behavior of the topological properties of sequences of fullerenes as the number of vertices $n$ grows.
The specific sequence of fullerenes chosen influences the limit, making the selection of its elements critical.
To apply statistical methods, assumptions about the distribution of fullerene selection and the dependencies between chosen elements are typically required.  
This naturally leads to the following problem:
\begin{align*}
    \text{How can one construct a uniform distribution on the elements of $C_n$?}
\end{align*}
\subsection{Naive approach}
A straightforward approach based on a given generating algorithm is as follows:
first, utilizing the given algorithm generate all fullerene isomers with $n$ vertices, and enumerate them from 1 to iso$(n)$.
Then, choose a uniformly random integer between 1 and iso$(n)$ to represent a uniformly chosen fullerene isomer. 

While theoretically sound, this method is practically limited by the computational difficulty of generating all fullerenes, in particular when $n > 400$. 
The high computational cost arises from the rapid growth of iso$(n)$, coupled with the computational costs of known generating algorithms.

\subsection{Generating algorithms for fullerenes}\label{section:generating_algorithms_for_fullerenes}
Many generating algorithms for fullerenes are based on \textit{growth operations} (or \textit{expansions}), and their inverse, \textit{reductions}. 
Growth algorithms typically begin with a small \textit{seed fullerene}, such as the dodecahedron $C_{20}$ or the $C_{24}$-isomer, the smallest fullerene with a non-zero number of hexagons.
These algorithms incrementally add hexagons, thereby increasing the number of vertices.
Notable examples of growth operations include those presented in \cite{Hash08}, which form the basis of the \textit{buckygen} software \cite{buckygen}; the operations defined by Buchstaber and Erokhovets \cite{BuchEro17}, which generalize the \textit{Endo-Kroto operations} from \cite{EK92} and are further discussed in a broader framework in \cite{Buchstaber_2017,Erokh18}; the \textit{leapfrog operations} described by Andova et al. \cite{AndKarSkr16}; and Coxeter's operations outlined in \cite{FT9928803117}. 
Importantly, Brinkman et al. \cite{Br09} demonstrated that no finite set of \textit{closed} growth operations can generate all possible fullerene isomers. 
Here, an operation is called \textit{closed} if its application to a fullerene always results in another fullerene.

It is worth noting that the growth operations proposed by Buchstaber et al. \cite{BuchEro17} form a finite family sufficient to construct all fullerenes starting from the dodecahedron $C_{20}$.
However, these operations are not closed, as they may produce intermediate graphs with one exceptional face, a quadrangle or a heptagon, before ultimately yielding a fullerene. 
If an algorithm can generate all possible fullerenes, we refer to it as \textit{complete}.

While growth algorithms are efficient for generating fullerenes with a small number of vertices ($n\leq 400$), their performance declines significantly as $n$ increases. 
This inefficiency is exacerbated by the fact that growth operations are not injective, meaning multiple operation sequences can lead to the same fullerene, resulting in a high rejection rate and increased computational costs.

Alternatively, \textit{isomerization algorithms} preserve the number of vertices while rearranging the existing pentagons and hexagons within the fullerene. 
The most well-known isomerization algorithm is based on the Stone--Wales operations, initially introduced in \cite{SW}.
Another example is the set of rotation and mirror operations on the fullerene graph, as introduced by Astakhova and Vinogradov \cite{ASTAKHOVA1998259}.

While more computationally efficient than growth algorithms, isomerization algorithms rely on the availability of a seed fullerene in $C_n$ beforehand. 
One approach to generate such seed fullerenes is through growth algorithms. 
Another way involves dividing the set of all feasible $n$ into congruence classes and, for each of these classes, defining a scalable fullerene structure -- such as a $(p,q)$--nanotube with a chiral vector depending on the congruence class -- that can be expanded to every element of the congruence class based on its vertex count. 
For example, for all $n=20+10r$ with $r\in\N_0$, a $(5,0)$--nanotube can be constructed similarly to the $(5,5)$--nanotubes described in \cite{bille24}.
In the case of $(5,0)$--nanotubes, the cap is composed exclusively of six vertices of degree 5, corresponding to six pentagons in the direct fullerene graph. 
Consequently, the hexagonal belt contains five vertices, compared to ten vertices in the $(5,5)$--nanotubes.
Another method for generating seed fullerenes for every $n\geq 34$ is discussed later in this section.

The \textit{face spiral algorithm}, one of the earliest methods for generating fullerene structures, is  neither a growth algorithm nor an isomerization algorithm.
Originally proposed by Manolopoulos et al. \cite{ManMayDown91}, a generalized approach - proven to be complete - was later presented in \cite{WSA17} and implemented in \cite{PeSchwertFull}.

\subsection{Face spiral acceptance-rejection algorithm}
While implementation \cite{PeSchwertFull} represents the most efficient version of the spiral algorithm, it is significantly outperformed by buckygen, currently the most efficient algorithm for fullerene construction. Nevertheless, the spiral algorithm is valuable for its potential to generate a uniform distribution within $C_n$.

The main idea behind the algorithm is to peel a fullerene like an orange, starting from one face and moving through each successive face in a tight spiral manner. 
As this spiral progresses, one records whether each face is a pentagon or a hexagon, producing a sequence of 5's and 6's. 
By repeating this process in each of the possible $6n$ ways (since a $C_n$-isomer has $3n$ facets as potential starting poits of the spiral, and two possible directions for each facet), a list of many sequences is generated. 
The lexicografically smallest spiral is then selected. 
The positions of the twelve 5's in this spiral provide the locations of the pentagonal faces yielding a final sequence of length 12 representing the fullerene.

Listing fullerenes in $C_n$ in lexicographical order, based on the positions of their pentagons within the spiral, provides a compact and interpretable representation and enumeration of fullerenes. 
To randomly select a fullerene from $C_n$ (each isomer having the same probability of being chosen), one can choose a random integer $N$ between 1 and $\binom{\sfrac{n}{2} + 2}{12}$ (the total number of all possible positions of pentagons in ascending order),  and then retrieve the $N^{th}$ sequence from the set of all position vectors. 

While generating these sequences can be efficiently realized, their number of length $n$ with twelve 5's as coordinates compared to the relatively small number $\mathcal{O}(n^9)$ of fullerene isomers results in a high rejection rate, making this approach impractical. 
For instance, for $n=60$, there are $\binom{32}{12} = 1,399,358,844,975$ possible pentagon position sequences, but only 1812 isomers in $C_{60}$. 
This results in an acceptance rate of approximately $10^{-9}$. 
As the binomial coefficient $\binom{\sfrac{n}{2}+2}{12}$ grows asymptotically as $n^{12}$ due to \textit{Stirling's formula} \cite{abra}, and therefore, faster than iso$(n)=\mathcal{O}(n^9)$ as $n\to\infty$, the acceptance rate diminishes even further for larger $n$.

The first known fullerene to violate this spiral rule is a $C_{380}$--isomer, with a second counterexample found among more than 90 billion $C_{384}$-isomers. 
However, all other isomers of $C_n$ with $n\leq 450$ follow this rule, cf. \cite{ManFow}. 

Although the authors of \cite{WSA17} proposed clever methods to prefilter sequences which do not correspond to fullerenes, the acceptance rate remains too low to make this approach feasible for large $n$.

Despite these challenges, the spiral method is still widely used for enumerating fullerenes and other polyhedral structures, as it can be extended to all cubic polyhedral graphs, cf. \cite{WSA17}.
Therefore, when referring to $C_{n,j}$, we mean the $j^{th}$ isomer in the lexicographical order of all $C_n$-isomers, based on their smallest pentagonal position sequence. 

The introductory question of this section remains valid and can be reformulated as a problem statement in the following way:
\begin{openproblem}
    Develop a computationally efficient algorithm to generate a
    uniform distribution on fullerene isomers $C_n$ for all feasible $n$.
\end{openproblem}
The Stone--Wales operation, an isomerization operation discussed in the following section, has the potential to lead to a solution to this open problem.

\section{Stones--Wales Operation}\label{section:stone_wales}
In 1986, A. J. Stone and his student D. J. Wales \cite{SW} introduced a flip operation on four faces of fullerene graphs, later known as the \textit{Stone--Wales (SW) operation}.
This operation gained popularity not only among chemists but also among mathematicians and physicists.
In their work, Stone and Wales calculated the so-called \textit{$\pi$-electron-energy} for various $C_{60}$-isomers, generated by applying the SW operation to the Buckminster structure $C_{60,1812}$. 
SW operations are often referred to as \textit{Stone--Wales defects} in chemistry, and they are not limited to fullerenes, appearing in other carbon-based molecules as well \cite{BrayCastro15}.

Currently, no closed and complete isomerization algorithm is known, although several generalizations of the SW operation have attempted to combine both properties. 
Two such generalizations -- the \textit{polyhedral Stone-Wales} (pSW) and the \textit{generalized Stone-Wales} (gSW) operation -- are introduced in this section.
In both cases, the classical SW operation is presented as a special instance of these generalizations.

For the gSW operation, we prove that infinitely many fullerene isomers cannot be generated by any sequence of gSW operations, offering insights into the algorithm’s incompleteness and potential strategies for overcoming it. 

\color{black}
\subsection{Polyhedral Stone--Wales Operation}\label{section:polyhedral Stone-Wales Operation}
The pSW operation was introduced by Plestenjak et al. \cite{Plestenjak96} as part of an effort to generate fullerenes at random.
This operation modifies triangulations by selecting two adjacent vertices, $v_1$ and $v_2$, and the edge $(v_1,v_2)$ connecting them in $T_n$. 
Since $T_n$ is a triangulation, there are exactly two other vertices, $v_3$ and $v_4$, that are adjacent to both $v_1$ and $v_2$. 
The pSW operation removes the edge $(v_1,v_2)$ and adds a new edge connecting $v_3$ and $v_4$ akin to an edge flip. 
Figure \ref{fig:pSW} illustrates this process.
While the pSW operation maintains the number of vertices, it does not preserve the degrees of the vertices, meaning the resulting graph may no longer be a fullerene. 
Therefore, the pSW is not closed with respect to the set of fullerenes but rather operates on the whole set of triangulations with $\sfrac{n}{2}+2$ vertices. 

\begin{figure}
    \centering
    \includegraphics[width=0.5\linewidth]{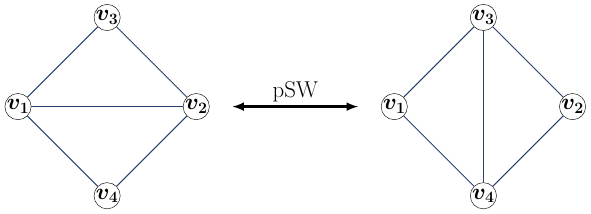}
    \caption{pSW operation on four vertices with arbitrary degrees.}
    \label{fig:pSW}
\end{figure}

Nevertheless, based on Wagner's result from \cite{Wagner1936}, this set of operations can eventually generate all $C_n$ isomers. 
Thus, the pSW operation is complete. 

The usual SW operation is a special case of the pSW, in which $v_1$ and $v_2$ are required to have degree six, while $v_3$ and $v_4$ must have degree five. 
In this case, the two hexagonal vertices become pentagonal, and the two pentagonal vertices become hexagonal after applying the pSW operation.

As noted by Plestenjak et al. in the conclusion of \cite{Plestenjak96}, it remains unclear how the choice of the seed fullerene affects the distribution of the generated fullerene isomers at the end of the algorithm.
In their article, they used the $\sfrac{n}{2}$-gonal prism, which can be constructed for any feasible $n$ and has the same number of vertices, edges and faces as an $C_n$-isomer.
The choice of vertices for applying the pSW operation is partly random and based on a specific energy function for polyhedra.
Broadly speaking, this selection rule aims to minimize the energy of the polyhedra under consideration. 
If multiple choices of edges result in the same energy level, an edge is chosen uniformly at random. 

\subsection{Generalized Stone--Wales Operation}\label{section:generalized Stone-Wales Operation}
In 1995, Babi\'{c} et al. \cite{generalized_SW} introduced another generalization of the Stone--Wales operation, simply termed \textit{generalized Stone--Wales} (gSW) operation, which applies to arbitrary large fragments of fullerenes with a specific structure. 

Let the length of a path on a graph be defined as $|V|-1$, where $V$ denotes the set of vertices in the path.
\begin{definition}\label{def:gSW}
For a feasible $n$ and $w\in\N$, $w\geq 2$, let a $C_n$-isomer with dual graph $T_n$ be given.
We call 
\begin{enumerate}[(a)]
\item a path $(v_1,\ldots,v_{2w})$ with $v_i\in V(T_n)$ a \textit{gSW path} of length $2w-1$, if the following in $T_n$ holds:
\begin{itemize}
\item $v_1$ and $v_{2w}$ have degree 5,
\item $v_2$ and $v_{2w-1}$ have degree 6, 
\item $v_i\not= v_j$ for $i\not= j$,
\item $(v_{i}, v_{i+2})\in E\left(T_n\right)$ for all $i=1,\ldots, 2w-2$.
\end{itemize}
\item a subgraph of $T_n$ induced by a gSW path, a \textit{gSW fragment}.
\item a function $F:C_n\rightarrow C_n$, which modifies the edges of $T_n$, such that $(v_2,v_1,v_4,v_3,\ldots, v_{2w},v_{2w-1})$ forms a gSW path in $F(T_n)$ if and only if $(v_1,v_2,\ldots, v_{2w})$ is a gSW path in $T_n$, a \textit{gSW operation}.
\end{enumerate}
\end{definition}

From this definition, it follows that the existence of a gSW path in a fullerene is necessary to apply the gSW operation to this fullerene.  
An example of a gSW path and gSW fragment with $w=5$ is shown in Figure \ref{fig:gSW operation}.
It is also worth noting that the gSW operation is self-inverse. 
The original SW operation is a special case of the gSW operation with $w=2$, meaning that no intermediate vertices exist between the first and last pair of vertices.
Since the pentagons and hexagons at the beginning and end of the gSW path in a fullerene switch places, and the degrees of all intermediate vertices remain the same, the gSW operation guarantees that the output is again a fullerene. 
Hence, the operation is closed with respect to the set $C_n$.

On one hand, the gSW operation extends the pSW operation, as both effectively involve edge flipping, though the pSW operation is restricted to flipping a single edge at a time. 
On the other hand, the gSW operation imposes specific degree constraints on the first and last two vertices in the gSW path.
In future studies, these constraints could be dropped to create a hybrid approach that combines both operations. 
However, an algorithm based on this combination would not be closed within the set $C_n$.
Nevertheless, this approach could lead to a significant reduction of computational costs.

\begin{figure}[H]
    \centering
    \includegraphics[scale=1]{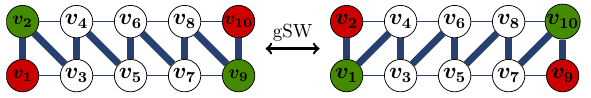}
    \caption{A gSW operation on a gSW path with ten vertices ($w=5$) with thick edges representing the gSW path necessary for the gSW operation.}
    \label{fig:gSW operation}
\end{figure}

The authors of \cite{generalized_SW} identified only one isomer with fewer than 70 vertices that lacks a gSW path. 
Based on this observation, they conjectured that gSW operations might be powerful enough to generate (almost) all fullerene isomers.  
However, due to this first counterexample, an algorithm solely based on the gSW operation clearly cannot be complete. 
The number and structure of fullerenes that cannot be generated by a gSW operation have yet to be fully classified, and a proper classification might provide insights into how to overcome this incompleteness.

We demonstrate that infinitely many isomers exist for which gSW operations cannot be applied. 
By computer experiments, we identified five counterexamples with fewer than 100 vertices, and two with fewer than 70 vertices. 
The specific counterexamples with $n\leq 100$ are: $C_{20}$, $C_{56,622}$, $C_{80,31924}$, $C_{92,39303}$ and $C_{96,191839}$. 
Since the authors of \cite{generalized_SW} briefly mention their counterexample without specifying the number of vertices, it is unclear which specific isomer they refer to. 
However, they likely exclude the dodecahedron $C_{20}$, as it is a trivial counterexample due to the abscence of pentagonal faces. 

In the following, we construct an infinite family of fullerenes that lack gSW paths, focussing only on paths on the subgraph $T_n^6$.
If we restrict the vertices $v_3,\ldots,v_{2w-2}$ to be elements in $V\left(T_n^6\right)$, a gSW path can be described as a zigzag path of odd length in $T_n^6$, starting and ending at boundary vertices, i.e., vertices with degree less than six.
Hence, we aim to define a structure for the connected components of $T_n^6$ that permits only zigzag paths of even length.

By studying the aforementioned counterexamples, we deduced a specific structure of the subgraph $T_n^6$ that does not permit a gSW path and remains consistent under scaling, allowing for the construction of an infinite set of counterexamples.
For $n\leq 100$, based on our numerical investigations, we claim that the five aforementioned counterexamples are the only isomers that do not allow for a gSW operation. 
For $n>100$, however, additional counterexamples with different structures may exist, posing an open problem for future research:
\begin{openproblem}
    Identify all fullerenes that do not permit a gSW operation.
\end{openproblem}

\noindent To provide a partly solution to this question, we need the following terms.
\begin{definition}\label{def:triangle} 
For a feasible $n$, let $T_n^6$ be the hexagonal dual graph of a fullerene in $C_n$, and let $t,r_1,r_2,r_3\in \N_0$, $r_1\geq r_2\geq r_3$,  $r_i+r_j\leq t-1$ for $i\not=j$.
We call a connected component of $T_n^6$
\begin{enumerate}[(a)]
\item a $t$\textit{-triangle}, for $t=0$, if it is a single vertex, and for $t>0$, if
\begin{itemize}
\item it has $t^2$ triangles as inner faces, and
\item it has $\frac{1}{2}(t+1)(t+2)$ vertices, three of which are of degree two, $3(t-1)$ vertices of degree four, and the remaining vertices of degree six. 
The vertices of degree two are called \textit{corners} of the $t$-triangle.
\end{itemize}

\item a $(t,(r_1,r_2,r_3))$\textit{-triangle}, if it is a $t$-triangle in which the vertices of $r_1$-, $r_2$-, and $r_3$-triangles at the corners have been deleted. 
Edges between vertices whose degree decreased due to the deletion are called \textit{open edges}.
\end{enumerate}
\end{definition}

\begin{figure}[H]
\centering
    \includegraphics[scale=1.2]{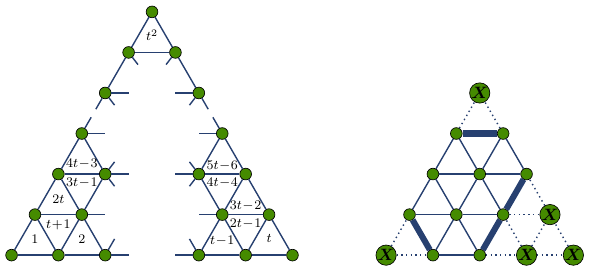}
\caption{Left: General structure of a $t$-triangle. 
Right: A $(4,(1,0,0))$-triangle with deleted edges (dotted lines) and deleted vertices (labeled $X$). 
Thick lines represent open edges.}
\label{fig:t_triangles}
\end{figure}

The general structure of $t$-triangles and an explicit example of a truncated $4$-triangle are illustrated in Figure \ref{fig:t_triangles}.

As it turns out the hexagonal subgraphs $T_n^6$ of every counterexample can be partitioned into a set of $t$-triangles and $(t,(r_1,r_2,r_3))$-triangles.
For the construction of this partition, we first need to introduce the following terms.

\begin{definition}
    A vertex $v\in V(T_n^6)$ is called a \textit{2-facet vertex} or \textit{3-facet vertex}, if $v$ is a boundary vertex of two or three facets of $T_n^6$ larger than a triangle, respectively. 
\end{definition}
From the definition, it follows immediately that a $2$-facet vertex has at least degree two and at most degree four, and that a $3$-facet vertex must have degree three. 
Further, since $T_n$ consists of triangles only, a 2-facet vertex of degree four is a vertex of exactly two triangles as illustrated in Figure \ref{fig:23_facet vertices}.

\begin{figure}[H]
    \centering
    \includegraphics[width=0.75\linewidth]{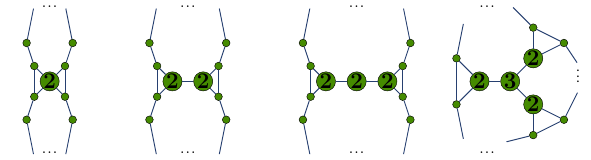}
    \caption{Cases of structures of $T_n^6$ with $2$-facet and $3$-facet vertices (labeled $2$ and $3$, respectively).}
    \label{fig:23_facet vertices}
\end{figure}

The main idea of the following construction is to ``cut'' the graph $T_n^6$ along its $2$-facet and $3$-facet vertices into $t$-triangles and $(t,(r_1,r_2,r_3))$-triangles, such that there are no $2$-facet and $3$-facet vertices and no vertices with degree 5 anymore.  


\begin{construction}\label{con:cutedge-partition} 
Let $T_n^6$ be the hexagonal subgraph of a dual fullerene $T_n$. 
Transform $T_n^6$ by applying the following two operations:
\begin{enumerate}[1.] 
\item For every $2$-facet vertex $v\in V(T_n^6)$, apply the following \textit{cut} operation:
\begin{enumerate}[(i)]
    \item Choose a minimal set of edges incident to $v$ such that deleting these edges makes $v$ a boundary vertex of only one facet larger than a triangle.
    \item After deleting these edges, insert a new vertex $\tilde{v}$ into the graph. 
    \item Reinsert the removed edges so that $\tilde{v}$ becomes the endpoint of these edges, replacing $v$. 
\end{enumerate}  
Denote the resulting graph by $\widetilde{T_n^6}$.

\item Let $v_1,\ldots,v_l \in V\left(\widetilde{T_n^6}\right)$ be all vertices of degree $5$ in $\widetilde{T_n^6}$.
While vertices of degree $5$ exists, perform the following procedure:
    \begin{enumerate}[(i)]
    \item Find a pair of vertices $(v_i,v_j)$, $i,j=1,\ldots,l$, $i\not=j$, on a connected component $CC$ of $\widetilde{T_n^6}$, such that the shortest path $p=(v_i,\ldots,v_j)$ between them is minimal among the shortest paths between any other pairs of vertices of degree $5$, and such that the edges of $p$ surround solely triangular facets. 
    The path $p$ divides $CC$ into two parts, denoted by $CC_1$ and $CC_2$. 
    \item Delete all edges with one endpoint in, say, $CC_2$, and the other endpoint being a vertex on the path $p$.
    \item For every vertex $v$ in $p$, add a new vertex $\widetilde{v}$ to $\widetilde{T_n^6}$.
    \item Add edges between these new vertices, replicating  the edges of the original path $p$.
    \item Reintroduce all edges deleted in step (ii), replacing the original endpoint $v$ with its corresponding copy $\widetilde{v}$.
    \end{enumerate}
\end{enumerate}
\end{construction}
\noindent Figure \ref{fig:cut_partition} illustrates both parts of Construction \ref{con:cutedge-partition} applied to the hexagonal subgraphs $T_{56,622}$, $T_{80,31924}$ and $T_{96,191839}^6$.

\begin{figure}[H]
    \centering
    \includegraphics[scale=1.2]{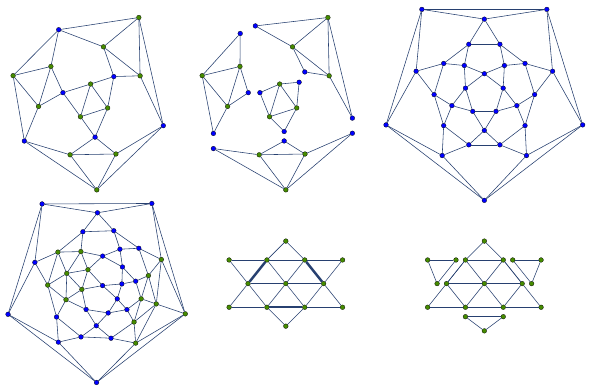}
    \caption{$T_{56,622}^6$ (upper left), $T_{80,31924}^6$ (upper right) and $T_{96,191839}^6$ (lower left) with $2$-facet vertices colored blue.
    The cut-partition of $T_{56,622}$ yields four $2$-triangles, and $T_{80,31924}$ yields 20 $1$-triangles.
    Applying the first part of Construction \ref{con:cutedge-partition} to $T_{96,191839}^6$ produces 12 $1$-triangles and two copies of the structure in the lower center. 
    Second part of the construction decomposes each of these copies into a $3$-triangle and three $1$-triangles. 
    Bold edges represent the path $\boldsymbol{p}$.}
    \label{fig:cut_partition}
\end{figure}

\begin{definition}
For a given $T_n^6$ we call the set of components of the graph generated following Construction \ref{con:cutedge-partition} a \textit{cut-partition} of $T_n^6$.
\end{definition}

\noindent
Based on these definitions and our numerical results, we now present the following conjectures.

\begin{conjecture}\label{conjecture:t-triangles_only}
    $C_{56,622}$ and $C_{80,31924}$ are the only fullerenes whose cut-partition consists of $t$-triangles only.
\end{conjecture}
    \noindent More importantly, we propose the following
\begin{conjecture}\label{conjecture:cut partitions no gSW}
    Let $T_n^6$, for a feasible $n$, be the hexagonal subgraph of $T_n$. 
    The cut-partition of $T_n^6$ consists solely of elements that are either $t$-triangles or $(t,(r_1,r_2,r_3))$-triangles if and only if $T_n$ contains no gSW path -- with the exception that the cut-partition contains only $0$-triangles.
\end{conjecture}

    \noindent These two conjectures now lead naturally to
\begin{openproblem}
    Prove or falsify Conjectures \ref{conjecture:t-triangles_only} and \ref{conjecture:cut partitions no gSW}.
\end{openproblem}

\noindent We now present an algorithm for constructing infinitely many fullerenes that do not contain a gSW path. 

\begin{construction}\label{ex:construction}
     For an integer $t\geq 2$, consider four $(2t,(t-1,t-1,t-1))$-triangles, glued together along their open edges to form a connected graph, which is the hexagonal subgraph $T_n^6$ of the final dual fullerene.
     Then, $T_n^6$ has four hexagonal and $t^2+6t-3$ triangular facets, and a total of $2(t^2+6t+2)$ vertices. 
     The degrees of the boundary vertices of the four hexagonal facets alternate between 4 and 5. 
     By inserting a $1$-triangle consisting of pentagonal vertices into each of these hexagonal facets, the structure ultimately forms a dual fullerene in $C_n$ with $n=4(t^2+6t+7)$.
\end{construction}
\begin{figure}[H]
    \begin{center}
    \includegraphics[scale=1.1]{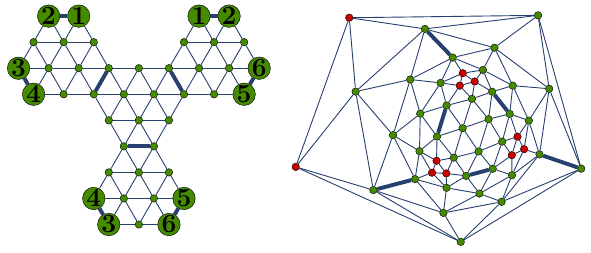}
    \end{center}
    \caption{Hexagonal graph $T_{92,39303}^6$ (left) and dual graph     
            $T_{92,39303}$ (right), where four $(4,(1,1,1))$-triangles have been glued together along the thick edges. 
            We identify vertices labeled with the same number.}\label{fig:t2_net}
\end{figure}
The key idea behind the proof of the following result is the observation that no dual fullerene constructed according to Construction \ref{ex:construction} allows for the existence of a gSW path.
\begin{theorem}\label{theorem:gSW}
    There exist infinitely many fullerenes on which a gSW operation cannot be applied. 
\end{theorem}
\begin{proof}
    Let $t\in\N$, $t\geq 2$. 
    For $n:=4(t^2+6t+7)$, let $T_n$ be obtained via Construction \ref{ex:construction}. 
    Denote by $T_n^5$ and $T_n^6$ the pentagonal and hexagonal subgraphs of $T_n$, respectively. 
    Here, we distinguish vertices in $T_n$ by their degree: vertices of degree five are labeled $p$, those with degree six are labeled $h$, and vertices whose degree is unspecified are labeled $v$.
    
    Now assume, for $w\in\N$, $w\geq 2$, that the dual fullerene graph $T_n$ contains a gSW path
    \begin{align*}
    \bm{p}:=(p_1,h_2,v_3,\ldots, v_{2w-2},h_{2w-1},p_{2w}),
    \end{align*}
    which is a zigzag path of odd length $2w-1$, a requirement for a gSW operation. 
    Since $p_1,p_{2w}\in V(T_n^5)$ and $h_2,h_{2w-1}\in V(T_n^6)$, the two vertices $h_2$ and $h_{2w-1}$ must be boundary vertices of two hexagonal facets of $T_n^6$.
    We now consider two cases, which are illustrated in Figure \ref{fig:gsW_proof}: 

    \textbf{1. Case:} 
    All the intermediate vertices $v_3,\ldots,v_{2w-2}$ of $\bm{p}$ are vertices of $T_n^6$, i.e., $v_j\in V(T_n^6)$ for every $j=3,\ldots,2w-2$.
    Then the truncated zigzag path $(h_2,\ldots,h_{2w-1})$ lies entirely in $T_n^6$, and its length $2w-3$ is odd. 
    However, by construction, any zigzag path starting and ending at boundary vertices of two distinct hexagonal facets in $T_n^6$ has even length. 
    This can be verified for $t=2$ in Figure \ref{fig:t2_net}.
    Hence, the gSW path cannot exist in this case, and no gSW operation can be applied. 
    \begin{figure}
        \centering
        \includegraphics[width=0.5\linewidth]{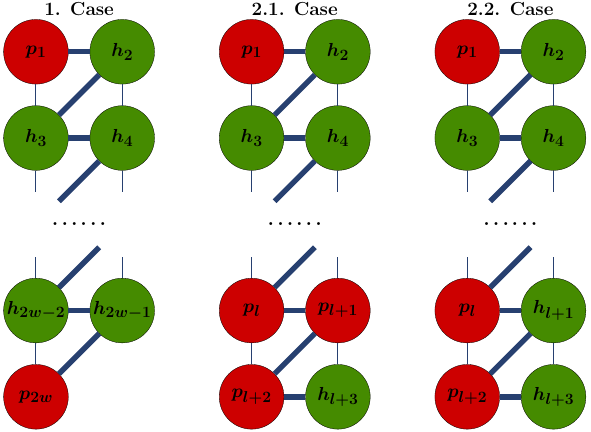}
        \caption{A zigzag path starting with a vertex of degree $5$ followed by a vertex of degree $6$ in $T_n$ (resulting from Construction \ref{ex:construction}) can never be a gSW path.}
        \label{fig:gsW_proof}
    \end{figure}
    
    \textbf{2. Case:} 
    There exists a minimal index $l \in \{3, 2w-2\}$ such that $v_l \in V(T_n^5)$, meaning that vertex $v_l$ has degree 5 in $T_n$.
    In this case we label it $p_l$ instead of $v_l$.

    Since the segment $(h_2,h_3,\ldots,h_{l-1})$ lies entirely on $T_n^6$, it has even length $l-3$, so $l$ must be odd.
    Given that $T_n^5$ consists of four connected components, each forming a $1$-triangle (triangular facet), we now consider the following two subcases depending on the degree of the next vertex in the path: 
    
    \textbf{2.1. Case:} 
    If the next vertex $p_{l+1}$ has also degree $5$, then $v_{l+2}$ must also have degree $5$, as the vertices in $T_n^5$ are arranged in $1$-triangles, and $T_n$ has solely triangular facets. 
    The zigzag path $(p_1,h_2,\ldots,p_l,p_{l+1},p_{l+2})$ is then of even length $l+1$.
    The remaining part of the path $(p_{l+2},h_{l+3},\ldots,,h_{2w-1},p_{2w})$ is also of even length $2w-l+5$, contradicting the requirement that a  gSW path must have odd length.
    Hence, $\bm{p}$ cannot be a gSW path. 
    If additional intermediate vertices $p_{l_1},\ldots,p_{l_u}$ of $\bm{p}$ belong to $V(T_n^5)$, the same reasoning can be applied inductively, since the segment $(p_{l_j+2},h_{l_j+3},\ldots,p_{l_{j+1}})$ has the same structure as $(p_1,h_2,\ldots,p_{l_1})$.

    \textbf{2.2. Case:} If the next vertex $h_{l+1} \in V(T_n^6)$, then $v_{l+2}$ must belong to $T_n^5$, and due to the structure of $T_n^5$, the next vertex must have degree $6$. 
    The segment of $\bm{p}$ starting from $p_{l+2}$ has the same structure as $(p_1,h_2,\ldots,p_{l_1})$ from the previous subcase, leading to the same contradiction.

    In all cases, we deduce that no gSW path can exist in $T_n$.
    Therefore, the gSW operation cannot be applied, proving the theorem.

\end{proof} 

Finally, let us discuss a simple construction of seed fullerenes in $C_n$ with $n\geq 36$ which may be used as a starting point of some isomerization algorithm.
Upon analyzing our data, we observed that for every $36\leq n\leq 150$, there exists an isomer whose pentagonal arrangement resembles the structure shown in Figure \ref{fig:C_36}. 
Indeed, it can be proved inductively that such a fullerene exists for any $n\geq 36$. 
The construction outlined below forms the main idea for proving this, and it also suggests how an efficient algorithm for generating a seed fullerene for a given $n\geq 36$ might be structured.

\begin{construction}\label{construction:seed fullerene}
Starting with $C_{36,1}$, remove all dotted edges shown on the left in Figure \ref{fig:C_36}. 
Label each boundary vertex $v$ of the larger connected component (the one with more vertices) with $6-d(v)$, where $d(v)$ is the degree of vertex $v$. 
This results in the lexicographically smallest label vector $(1,2,2,2,2,3)$ for the six consecutive boundary vertices. 
Next, by adding a new vertex and connecting it to the vertex labeled 1 and its adjacent vertices labeled $2$ and $3$, a connected component is obtained with an additional vertex while preserving the same label vector as shown on the right in Figure \ref{fig:C_36}.
Afterwards, the dotted edges can be reintroduced, resulting in a $C_{38}$ fullerene.

This procedure can be iterated, as the label vector and its length remain unchanged throughout. 
This implies that the modified, larger connected component can be reconnected with the other, smaller connected component. 
By repeating the process $\sfrac{n}{2}-18$ times, a $C_n$-isomer for any feasible $n\geq 36$ can be efficiently constructed. 
\end{construction}

 \begin{figure}[H]
\begin{center}
\includegraphics[scale=1.1]{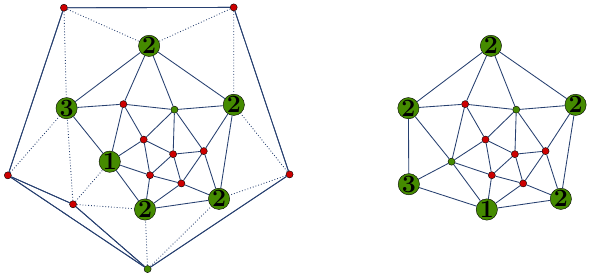}
\end{center}
\caption{Left: Dual fullerene graph $T_{36,1}$. Isomer $C_{36,1}$ can serve as an initial structure 
        for generating seed fullerenes of arbitrary size for any isomerization algorithm. Right: The inner connected component obtained by applying Construction \ref{construction:seed fullerene} to $C_{36,1}$ once with the unchanged boundary degree sequence preserved. 
        Replacing the corresponding component in $C_{36,1}$ with this component yields a dual fullerene graph in $C_{38}$.}\label{fig:C_36}
\end{figure}

\section{Spectra of fullerenes}\label{section:spectra_of_fullerenes}

Next, we aim to describe the shape of a fullerene using the spectra of its adjacency and degree matrices.
The spectral analysis of graphs based on various matrix representations has a long-standing tradition, as seen in \cite{BroHae,BruCve,Bap}.

\subsection{Characters as unique fullerene descriptors}

In \cite{Bille20}, the authors uncovered a notable connection between the (chemical) relative energy of fullerene molecules $C_{60}$ and \textit{Newton polynomials of order} $k\in\N_0$, which are defined as

\begin{align*}
    N(M,k):= \tr(M^k) = \sum_{j=1}^n \lambda_j^k,
\end{align*}
where $\lambda_j$, $j=1,\ldots,n$, are the eigenvalues of a matrix $M$ representing a dual fullerene graph $T_n$. 

Common choices for $M$ include the adjacency matrix $A$, the degree matrix $D$, or a linear combination of those $\alpha A+\beta D$ with $\alpha,\beta\in\R$.
Alternatively, instead of the full dual fullerene graph $T_n$, one may focus on subgraphs such as $T_n^6$ or $T_n^5$, and their associated matrices.
For instance, the study in \cite{Bille20} focuses on the adjacency matrix of the hexagonal subgraph $T_n^6$, where the authors observed a Pearson correlation coefficient between Newton polynomials $N(M,k)$ and the relative energy of the molecules often exceeding $0.9$, depending on the order $k$.
This high correlation, along with the energy-based ordering of fullerenes outlined in \cite{Grimme17}, allows for the identification of energetically stable isomers at a reduced computational cost.
The authors chose hexagonal subgraphs partly because they conjectured that the positions and arrangements of the remaining twelve pentagons could be inferred from larger-than-triangle facets in $T_n^6$ and the degree sequences of their boundary vertices.

Notably, when $M=A$, the Newton polynomials $N(A,k)$ represent the number of closed paths of length $k$ on the graph with adjacency matrix $A$.
This interpretation reveals that Newton polynomials of varying orders $k$ encode specific structural information about the fullerene, such as the number of edges when $k=2$.
This suggests the potential to combine Newton polynomials across all orders into a comprehensive graph descriptor that uniquely characterizes the fullerene.

\begin{definition}
    For a feasible $n$, let $T_n$ be a dual fullerene graph, and $A$ and $D$ its adjacency and degree matrices. 
    For $\alpha, \beta \in\R$, we call 
    \begin{align*}
    ch_{\alpha,\beta}(T_n):=\tr\left( \exp\left( \alpha A + \beta D \right)\right)
    \end{align*}
    the $(\alpha,\beta)$--\textit{character} of $T_n$, where the matrix exponential is defined as
    \begin{align*}
    \exp(M):= \sum_{k=0}^\infty \frac{1}{k!}M^k
    \end{align*}
    for every complex-valued quadratic matrix $M$. 
\end{definition}
Note that the $(\alpha,\beta)$-character can be rewritten as
\begin{align*}
ch_{\alpha,\beta}(T_n)=\sum_{k=0}^\infty \frac{\alpha^k}{k!}\tr\left( \left(A+\frac{\beta}{\alpha}D \right)^k \right)=\sum_{j=1}^m e^{\alpha \lambda_j},
\end{align*}
where $\lambda_1,\ldots,\lambda_m$ with $m=\sfrac{n}{2}+2$, are the eigenvalues of $A+\sfrac{\beta}{\alpha}D$.
This representation reveals that the character is essentially an infinite series of Newton polynomials of the adjacency matrix $A+\sfrac{\beta}{\alpha}D$.  
Analyzing the eigenvalues of $A+\sfrac{\beta}{\alpha}D$ becomes sufficient to compute the $(\alpha,\beta)$-character.

The behaviour of the $(\alpha,\beta)$-character varies considerably with different choices of parameters $\alpha$ and $\beta$, raising the important question of identifying an appropriate range of parameters. 
Extensive numerical experiments for various combinations of $\alpha,\beta$ have led us to the following statement.

\begin{conjecture}\label{con:unique_character}
For $\alpha ,\beta\in (0,1)$ and feasible $n_1,n_2$, let $T_{n_1},T_{n_2}$ be two dual fullerenes. 
It holds that
\begin{align*}
T_{n_1}\simeq T_{n_2} 
\Leftrightarrow 
ch_{\alpha,\beta}\left(T_{n_1}\right)
=
ch_{\alpha,\beta}\left(T_{n_2}\right).
\end{align*}
\end{conjecture}

Therefore, the following is presented for future research:
\begin{openproblem}
    Prove or falsify Conjecture \ref{con:unique_character}.
\end{openproblem}

As the number of vertices in a fullerene tends to infinity, the fullerene structure becomes composed solely of hexagonal facets. 
Two notable configurations of infinitely many hexagons are the \textit{hexagonal lattice} (also referred to as \textit{graphene}) and \textit{infinite $(p,q)$--nanotubes}.
For precise definitions and further details on their eigenvalues, we refer to \cite{bille23} for the hexagonal lattice and \cite{bille24} for the infinite $(p,q)$--nanotubes.
In general, the dual of a 3-regular infinite graph $H$, composed only of hexagonal facets, is an infinite triangulation $T$.
If a loop of weight $3$ is added to every vertex of $T$, then the average number of closed paths of length $k$ per vertex, denoted by $\mu_k(\cdot)$, satisfies the following relation:
\begin{align}
    \mu_{2k}(H) = \mu_k(T),\quad k\in\N_0.
\end{align}

This observation connects to a similar result for finite fullerene graphs. 
As conjectured in \cite{bille23}, a sequence of dual fullerene graphs, represented by the matrix $A_n+\sfrac{1}{2}D_n$ can be constructed such that their normalized Newton polynomials of degree $k$ converge to $\mu_k(T)$ as $n\to \infty$:
\begin{align*}
    \frac{2}{n+4}N\left(A_n + \frac{1}{2}D_n,k \right)\xrightarrow[n\to \infty]{} \mu_k(T),\quad k\in\N.
\end{align*}
For their duals, i.e., the direct fullerene graphs represented by $A_n$, it follows that
\begin{align*}
    \frac{1}{n} N\left(A_n,k\right) \xrightarrow[n\to\infty]{} \mu_k(H),\quad k\in\N.
\end{align*}

These observations, along with Conjecture \ref{con:unique_character}, suggest the parameter choice $\alpha=2\beta=\sfrac{1}{2}$ when analyzing dual fullerene graphs, and $\alpha=1, \beta=0$ for direct fullerene graphs.
Henceforth, whenever the parameters $(\alpha,\beta)$ are omitted, we refer to the $\left( \sfrac{1}{2},\sfrac{1}{4}\right)$--character.
In addition to the unique characterization of fullerenes in Conjecture \ref{con:unique_character}, the characters seem to induce an ordering within the set of all fullerenes that is reflexive, transitive, antisymmetric, and total.  

\subsection{Linear ordering of fullerenes and their limiting shape}
\label{sec:linear_ordering_of_fullerenes}

Provided that  Conjecture \ref{con:unique_character} holds true, an enumeration method based on characters is injective, which is a necessary condition for linearity.
A second desirable property for enumeration is the monotonicity with respect to the number of vertices, meaning that

\begin{align}
ch_{\alpha,\beta}(T_{n_1})<ch_{\alpha,\beta}(T_{n_2}) \Leftrightarrow n_1 < n_2. \label{condition_monotonicity}
\end{align}
When both $\alpha$ and $\beta$ tend to 0, the $(\alpha,\beta)$--character converges to the number of vertices in the corresponding dual graph:

\begin{align*}
\lim_{\alpha\rightarrow 0}\lim_{\beta\rightarrow 0} ch_{\alpha,\beta}(T_n)=
\lim_{\beta\rightarrow 0}\lim_{\alpha\rightarrow 0} ch_{\alpha,\beta}(T_n)=\frac{n}{2}+2.
\end{align*} 
Thus, condition (\ref{condition_monotonicity}) is satisfied for sufficiently small $\alpha,\beta$. 
However, if $\alpha$ and $\beta$ are too small, the numerical precision required to distinguish between different fullerenes in $C_n$ increases, leading to higher computational costs, as the characters must be computed with greater accuracy.   

On the other hand, the character grows exponentially with $\alpha$ and $\beta$, which can result in the smallest character in $C_n$ becoming  smaller than the largest character in $C_{n-2}$. 
In other words, for a feasible $n$, the range of characters $ch_{\alpha,\beta}$ in $C_n$ expands with increasing $\alpha$ and $\beta$. 
Therefore, the parameters must be chosen small enough to prevent overlapping in the character ranges for different $n$, ensuring that condition (\ref{condition_monotonicity}) is maintained. 

These considerations highlight the importance of selecting appropriate values for $\alpha$ and $\beta$ to balance computational efficiency while ensuring that condition (\ref{condition_monotonicity}) is fulfilled. 

In Figure \ref{fig:linear_ordering}, histograms of the $(\alpha,\beta)$--characters are shown for four combinations of $(\alpha,\beta)$ and $n\in\lbrace 58,60,62 \rbrace$.
The vertical dotted lines mark the minimal and maximal character for each $C_n$, with blue, red, and green lines representing $C_{58}$, $C_{60}$, and $C_{62}$, respectively. 
In the upper-left plot with $\alpha=\beta=1$, the character ranges overlap, with the the largest character in $C_{58}$ not only surpassing all characters in $C_{58}$ by far but also the smallest character of $C_{60}$. 
This overlap is significantly reduced for $(\alpha,\beta)=(1,\sfrac{1}{2})$, as seen in the upper-right plot, though a small intersection remains between $C_{60}$ and $C_{62}$, caused by the appearance of the $(5,0)$--nanotube in $C_{60}$, which has an disproportionately large character.   
For $(\alpha,\beta)=(\sfrac{1}{2},1)$, the ranges no longer overlap, but the character values grow significantly, even for small $n$, as indicated by the ticks on the $x$-axis. 
Thus, computing characters and enumerating fullerenes with these parameter values would demand significant computational power as $n$ increases. 
The lower-right plot shows more distinct character ranges for the chosen values $\alpha=2\beta=\sfrac{1}{2}$ with a more practical character scale, supporting the choice of parameters in this study.

Additionally, extrem fullerene structures were observed to have the largest and smallest characters. 
In particular, $(p,q)$-nanontubes typically exhibit higher characters, where smaller circumferences $p+q$ correspond to larger characters.
This relationship arises from the fact that nanotubes with smaller circumferences permit a greater number of closed paths of a given length resulting in a higher character. 
For example, the $(5,0)$--nanotube provides an upper bound for the character within $C_n$. 
As detailed in \cite{bille24} and as discussed in Section \ref{section:generating_algorithms_for_fullerenes}, a $(5,0)$--nanotube exists for every $n = 20 + 10r$ with $r\in\N_0$.

In contrast, \textit{Goldberg polyhedra}, i.e., fullerenes with icosahedral symmetry (highest symmetry a fullerene can attain), tend to have the smallest characters. 
A Goldberg polyhedron exists in $C_n$ if two integers $p,q\in\N_0$ satisfy $n=20\left((p+q)^2 -pq\right)$, cf. \cite{Fan17}, making the Goldberg polyhedron a nanotube with chiral vector $(5p,5q)$, cf. \cite{bille24}, corresponding to the largest feasible circumference in $C_n$.

\begin{figure}
    \centering
    \includegraphics[width=1\linewidth]{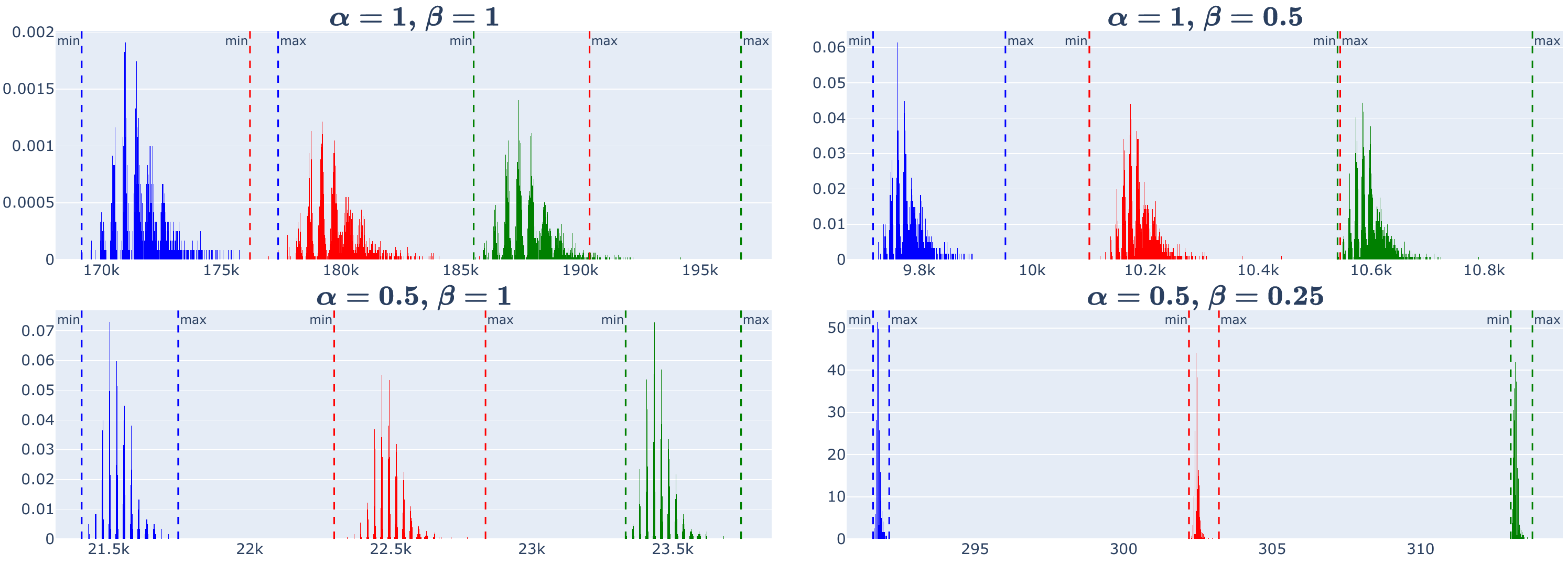}
    \caption{Ranges of $(\alpha,\beta)$--characters of $C_{58}$ (blue), $C_{60}$ (red), $C_{62}$ (green) with different parameters $(\alpha,\beta)$.}
    \label{fig:linear_ordering}
\end{figure}

In cases where a feasible $n$ allows integers $r,p,q\in\N_0$ to satisfy the equation $n=20+10r=20\left((p+q)^2 -pq\right)$, meaning that $C_n$ contains both a Goldberg polyhedron and a $(5,0)$--nanotube (structures with extreme characters), these two can be utilized to centralize and normalize the characters of all dual fullerenes within $C_n$. 

A fundamental question that arises here is how these centralized and normalized characters are distributed over the interval $[0,1]$. 
\begin{openproblem}
    Check the convergence of 
    \begin{align}
        \frac{ch_{\sfrac{1}{2},\sfrac{1}{4}}(T_n) - \min_{G}ch_{\sfrac{1}{2},\sfrac{1}{4}}(G)}{\max_{G}ch_{\sfrac{1}{2},\sfrac{1}{4}}(G) -\min_{G}ch_{\sfrac{1}{2},\sfrac{1}{4}}(G)}
        \label{formula:limit_centra_norma_character}
    \end{align}
    as $n\to \infty$, where $G$ represents a dual fullerene graph in $C_n$, and $T_n$ a random dual fullerene graph in $C_n$. 
    If the limit exists, find its explicit form.
\end{openproblem}

One possible starting point for addressing this problem is Figure \ref{fig:character_center_norm_T_T6_60_80_140}, which shows the empirical densities of characters adjusted according to \eqref{formula:limit_centra_norma_character} of all $C_n$-isomers for $n=60,80,140$. 
It appears that as $n$ grows, the normalized histograms tend to converge to a $\Gamma$-shaped limiting distribution density.

Understanding the asymptotic behavior of expression \eqref{formula:limit_centra_norma_character} could provide insights into generating large random fullerene isomers by sampling from the limiting distribution of their centralized and normalized characters. 
\color{black}

Since dual fullerenes consist solely of vertices of degree six in the limit $n\to \infty$, it may be advantageous to analyze the same distribution using the hexagonal subgraph $T_n^6$ instead of $T_n$. 

\begin{figure}[H]
    \centering
    \includegraphics[width=0.9\linewidth]{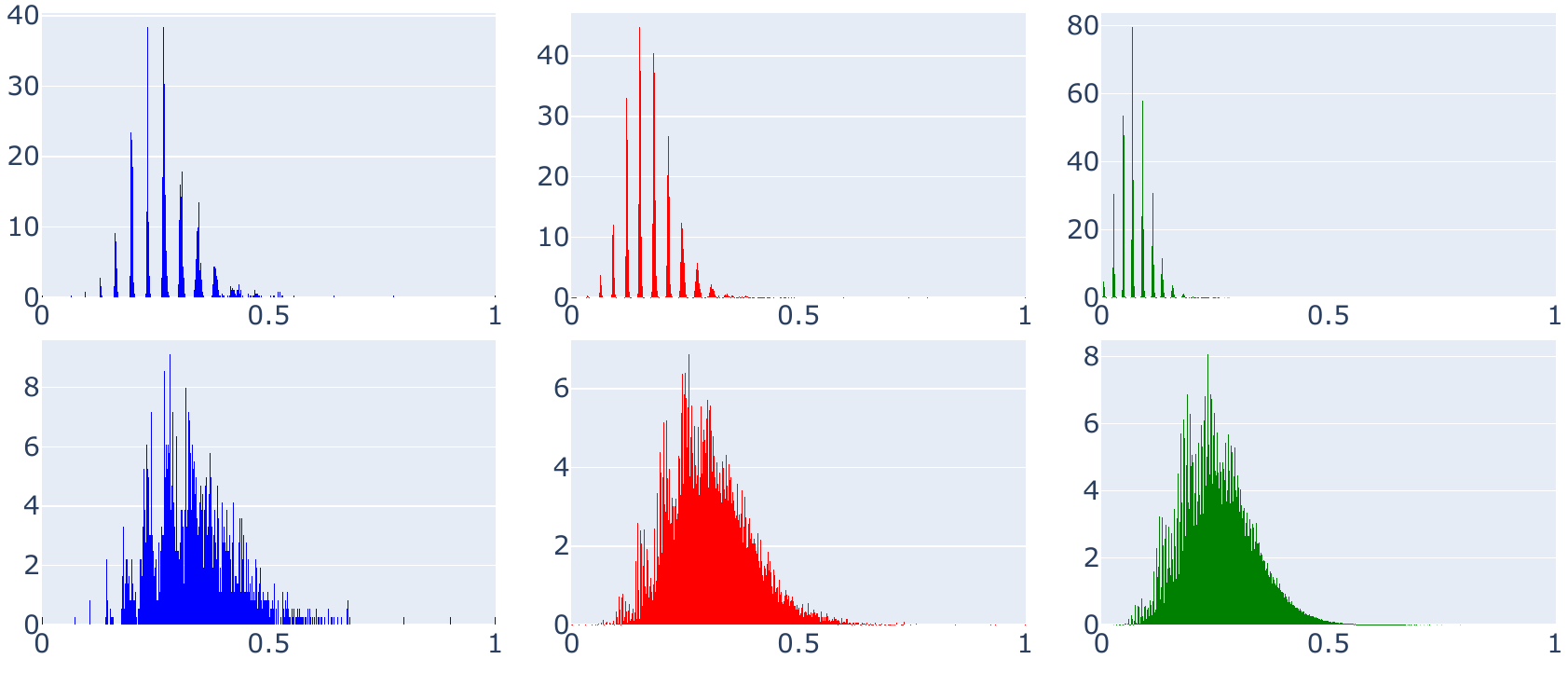}
    \caption{Normalized histograms (1000 bins) of characters of isomers $C_n$, centered and normalized as per \eqref{formula:limit_centra_norma_character}, for $n=60$ (blue), $n=80$ (red) and $n=140$ (green) using $T_n$ (upper row) and $T_n^6$ (lower row).}
    \label{fig:character_center_norm_T_T6_60_80_140}
\end{figure}

\section{Summary}\label{section:summary}
This paper explores several open mathematical problems related to fullerenes, with a primary focus on their generation, enumeration, and spectral properties. 
It articulates several key problem statements and conjectures, summarized at the end, which call for further investigation.

The paper begins by tackling the challenge of generating random fullerene isomers, offering an overview of current generating algorithms. 
While the face spiral conjecture implies an established method for enumerating fullerenes, it fails to yield a uniform distribution over $C_n$, leaving open the problem of how to achieve such uniformity.

Next, the discussion shifts to a generating algorithm based on the gSW operation, revealing an infinite family of fullerenes for which this operation is proven to be inapplicable. 
Numerical examples illustrate the structural limitations of  the gSW operation, particularly for fullerenes with fewer than 100 vertices.

Additionally, the introduction of a novel graph invariant for fullerenes, termed the \textit{character}, derived from  adjacency and degree matrices, marks another significant contribution.
This invariant conjectures a linear ordering of all fullerenes, offering an alternative framework to the face spiral method for enumeration and structural analysis of fullerenes.
Conjecture \ref{con:unique_character} and the investigation of the limiting behavior of the centralized and normalized characters, as given in expression \eqref{formula:limit_centra_norma_character}, pose fundamental questions regarding the uniqueness of the character and its applicability to fullerene ordering and generation.

To summarize, this paper outlines the following five open problems that might motivate future research:

\begin{enumerate}
    \item Develop a computationally efficient algorithm to generate a
        uniform distribution on fullerene isomers $C_n$ for all feasible $n$.
    \item Identify all fullerenes that do not permit a gSW 
        operation.
    \item Prove or falsify Conjectures \ref{conjecture:t-triangles_only} 
        and \ref{conjecture:cut partitions no gSW} regarding the relationship between cut-partitions and the gSW operation. 
    \item Prove or falsify Conjecture \ref{con:unique_character} about 
        the uniqueness of $(\alpha,\beta)$--characters.
    \item Determine, if existent, the distribution of centered and 
        normalized characters \eqref{formula:limit_centra_norma_character} of random fullerenes as $n\to \infty$. 
\end{enumerate}

\subsection*{Acknowledgement}
We express our gratitude to Annika Braig and Christian Spaete for their implementation of the algorithms and computation of the isomerization maps for the generalized Stone-Wales operation.

\newpage

\bibliography{Literature}{}

\begin{thebibliography}{10}

\bibitem{abra}
M.~Abramowitz and I.~A. Stegun.
\newblock {\em Handbook of Mathematical Functions with Formulas, Graphs, and Mathematical Tables}.
\newblock Dover, 1964.

\bibitem{AndKarSkr16}
V.~Andova, F.~Kardo{\v s}, and R.~{\v S}krekovski.
\newblock Mathematical aspects of fullerenes.
\newblock {\em Ars Mathematica Contemporanea}, 11:353--379, 2016.

\bibitem{ASTAKHOVA1998259}
T.~Y. Astakhova and G.~A. Vinogradov.
\newblock New isomerization operations for fullerene graphs.
\newblock {\em Journal of Molecular Structure: THEOCHEM}, 430:259--268, 1998.

\bibitem{generalized_SW}
D.~Babic, S.~Bassoli, M.~Casartelli, F.~Cataldo, A.~Graovac, O.~Ori, and B.~York.
\newblock Generalized {S}tone-{W}ales transformations.
\newblock {\em Molecular Simulation}, 14:395--401, 1995.

\bibitem{Bap}
R.~B. Bapath.
\newblock {\em Graphs and matrices}.
\newblock Springer, 2010.

\bibitem{bille_github_phd}
A.~Bille.
\newblock Github software repository on fullerenes.
\newblock \url{https://github.com/fullereneUulm/classification_of_combinatorially_equivalent_polytopes_thesis.git}, Accessed: \today.

\bibitem{bille23}
A.~Bille, V.~Buchstaber, S.~Coste, S.~Kuriki, and E.~Spodarev.
\newblock Random eigenvalues of graphenes and the triangulation of plane.
\newblock arXiv preprint No 2306.01462, submitted, 2023.

\bibitem{bille24}
A.~Bille, V.~Buchstaber, P.~Ievlev, S.~Novikov, and E.~Spodarev.
\newblock Random eigenvalues of nanotubes.
\newblock arXiv preprint No 2408.14313, submitted, 2024.

\bibitem{Bille20}
A.~Bille, V.~Buchstaber, and E.~Spodarev.
\newblock Spectral clustering of combinatorial fullerene isomers based on their facet graph structure.
\newblock {\em Journal of Mathematical Chemistry}, 59:264--288, 2020.

\bibitem{BrayCastro15}
E.~Brayfindley, E.~E. Irace, C.~Castro, and W.~L. Karney.
\newblock Stone–{W}ales rearrangements in polycyclic aromatic hydrocarbons: A computational study.
\newblock {\em The Journal of Organic Chemistry}, 80:3825--3831, 2015.

\bibitem{buckygen}
G.~Brinkmann, J.~Goedgebeur, and B.~McKay.
\newblock The generation of fullerenes.
\newblock {\em Journal of chemical information and modeling}, 52:2910--2918, 2012.

\bibitem{Br09}
G.~Brinkmann, J.~E. Graver, and C.~Justus.
\newblock Numbers of faces in disordered patches.
\newblock {\em Journal of Mathematical Chemistry}, 45:263--278, 2009.

\bibitem{BroHae}
A.~E. Brouwer and W.~H. Haemers.
\newblock {\em Spectra of graphs}.
\newblock Springer, 2012.

\bibitem{BruCve}
R.~A. Brualdi and D.~Cvetkovic.
\newblock {\em A combinatorial approach to matrix theory and its applications}.
\newblock Chapman \& Hall/CRC Press, 2009.

\bibitem{BuchEro17}
V.~Buchstaber and N.~Erokhovets.
\newblock Finite sets of operations sufficient to construct any fullerene from {$C_{20}$}.
\newblock {\em Structural Chemistry}, 28:225--234, 2017.

\bibitem{Buchstaber_2017}
V.~M. Buchstaber and N.~Y. Erokhovets.
\newblock Constructions of families of three-dimensional polytopes, characteristic patches of fullerenes, and {P}ogorelov polytopes.
\newblock {\em Izvestiya: Mathematics}, 81:901--972, 2017.

\bibitem{HouseofGraphs}
K.~Coolsaet, S.~D’hondt, and J.~Goedgebeur.
\newblock House of graphs 2.0: A database of interesting graphs and more.
\newblock {\em Discrete Applied Mathematics}, 325:97–107, 2023.

\bibitem{EK92}
M.~Endo and H.~W. Kroto.
\newblock Formation of carbon nanofibers.
\newblock {\em The Journal of Physical Chemistry}, 96:6941--6944, 1992.

\bibitem{engel2024}
P.~Engel, J.~Goedgebeur, and P.~Smillie.
\newblock Exact enumeration of fullerenes, 2024.
\newblock arXiv preprint No 2304.01655.

\bibitem{engel2018}
P.~Engel and P.~Smillie.
\newblock The number of convex tilings of the sphere by triangles, squares, or hexagons.
\newblock {\em Geometry \& Topology}, 22:2839--2864, 2018.

\bibitem{erman09}
R.~Erman, F.~Kardo{\v{s}}, and J.~Mi{\v{s}}kuf.
\newblock Long cycles in fullerene graphs.
\newblock {\em Journal of mathematical chemistry}, 46:1103--1111, 2009.

\bibitem{Erokh18}
N.~Y. Erokhovets.
\newblock Construction of fullerenes and {P}ogorelov polytopes with 5--, 6-- and one 7--gonal face.
\newblock {\em Symmetry}, 10:67--95, 2018.

\bibitem{Fan17}
Y.-J. Fan and B.-Y. Jin.
\newblock From the "brazuca" ball to octahedral fullerenes: Their construction and classification.
\newblock {\em Journal of Mathematical Chemistry}, 55, 03 2017.

\bibitem{FowlerManop}
P.~W. Fowler and D.~E. Manolopoulos.
\newblock {\em An Atlas of Fullerenes}.
\newblock Clarendon Press, 1995.

\bibitem{Grunbaum}
B.~Gr{\"u}nbaum.
\newblock {\em Convex Polytopes}.
\newblock Springer, 2003.

\bibitem{Grünbaum_Motzkin_1963}
B.~Grünbaum and T.~S. Motzkin.
\newblock The number of hexagons and the simplicity of geodesics on certain polyhedra.
\newblock {\em Canadian Journal of Mathematics}, 15:744–751, 1963.

\bibitem{Hash08}
M.~Hasheminezhad, H.~Fleischner, and B.~D. McKay.
\newblock A universal set of growth operations for fullerenes.
\newblock {\em Chemical Physics Letters}, 464:118--121, 2008.

\bibitem{computer_proof}
F.~Kardo\v{s}.
\newblock A computer-assisted proof of the barnette--goodey conjecture: Not only fullerene graphs are hamiltonian.
\newblock {\em SIAM Journal on Discrete Mathematics}, 34(1):62--100, 2020.

\bibitem{ManFow}
D.~E. Manolopoulos and P.~W. Fowler.
\newblock A fullerene without a spiral.
\newblock {\em Chemical Physics Letters}, 204:1--7, 1993.

\bibitem{FT9928803117}
D.~E. Manolopoulos, P.~W. Fowler, R.~Taylor, H.~W. Kroto, and D.~R.~M. Walton.
\newblock {Faraday communications. {A}n end to the search for the ground state of $C_{84}$?}
\newblock {\em Journal of the Chemical Society, Faraday Transactions}, 88:3117--3118, 1992.

\bibitem{ManMayDown91}
D.~E. Manolopoulos, J.~C. May, and S.~E. Down.
\newblock {Theoretical studies of the fullerenes: $C_{34}$ to $C_{70}$}.
\newblock {\em Chemical Physics Letters}, 181:105--111, 1991.

\bibitem{Plestenjak96}
B.~Plestenjak, T.~Pisanski, and A.~Graovac.
\newblock Generating fullerenes at random.
\newblock {\em Journal of Chemical Information and Computer Sciences}, 36:825--828, 1996.

\bibitem{Rukh18}
A.~Rukhovich.
\newblock On the growth rate of the number of fullerenes.
\newblock {\em Russian Mathematical Surveys}, 73:734--736, 2018.

\bibitem{PeSchwertFull}
P.~Schwerdtfeger, L.~Wirz, and J.~Avery.
\newblock Program {F}ullerene - a software package for constructing and analyzing structures of regular fullerenes.
\newblock {\em Journal of Computational Chemistry}, 34:1508--1526, 2013.

\bibitem{SW}
A.~J. Stone and D.~Wales.
\newblock Theoretical studies of icosahedral c60 and some related species.
\newblock {\em Chemical Physics Letters}, 128:501--503, 1986.

\bibitem{Grimme17}
R.~Sure, A.~Hansen, P.~Schwerdtfeger, and S.~Grimme.
\newblock {Comprehensive study of all 1812 $C_{60}$ isomers}.
\newblock {\em Physical Chemistry Chemical Physics}, 19:14296--14305, 2017.

\bibitem{Thurs98}
W.~P. Thurston.
\newblock Shapes of polyhedra and triangulations of the sphere.
\newblock In {\em The {E}pstein birthday schrift}, volume~1 of {\em {Geometry \& Topology Monographs}}, pages 511--549. {Geometry \& Topology Publications}, 1998.

\bibitem{Wagner1936}
K.~Wagner.
\newblock Bemerkungen zum {V}ierfarbenproblem.
\newblock {\em Jahresbericht der Deutschen Mathematiker-Vereinigung}, 46:26--32, 1936.

\bibitem{WSA17}
L.~N. Wirz, P.~Schwerdtfeger, and J.~E. Avery.
\newblock Naming polyhedra by general face-spirals --{T}heory and applications to fullerenes and other polyhedral molecules.
\newblock {\em Fullerenes, Nanotubes and Carbon Nanostructures}, 26:607--630, 2018.

\end{thebibliography}
\bibliographystyle{abbrv}


\end{document}